\documentclass{amsart}

\usepackage{amsmath, amsthm, amssymb}
\usepackage{graphicx}
\usepackage{tikz, tikz-cd}
\usetikzlibrary{arrows, matrix}
\usepackage{ytableau}
\usepackage{subfigure}
\usepackage[hidelinks]{hyperref}

\newcommand{\sg}{\sigma}

\def\des{\operatorname{des}}
\def\Des{\operatorname{Des}}
\def\asc{\operatorname{asc}}
\def\Asc{\operatorname{Asc}}
\def\inv{\operatorname{inv}}
\def\Inv{\operatorname{Inv}}
\def\maj{\operatorname{maj}}
\def\minimaj{\operatorname{minimaj}}

\def\Obj{\operatorname{Obj}}
\def\stat{\operatorname{stat}}

\def\dist{D}
\newcommand{\dists}[3] {D^{#1}_{#2}(#3)}
\newcommand{\distosp}[4] {D^{#1}_{#2, #3}(#4)}

\def\rlmaj{\operatorname{rlmaj}}

\newcommand{\qbinom}[2] { {#1 \brack #2}_{q} }

\newcommand{\qint}[1] { [#1]_{q} }

\newcommand{\qstir}[3] {S_{#1, #2}(#3)}
\newcommand{\allqstir}[4] {S_{#1, #2}^{(#4)}(#3)}

\newcommand{\osp}[2] {\mathcal{OP}_{#1, #2}}

\newcommand{\ospn}[1] {\mathcal{OP}_{#1}}

\newcommand{\op}[1] {#1}
\def\S{\mathfrak{S}}

\def\multiset#1#2{\ensuremath{\left(\kern-.3em\left(\genfrac{}{}{0pt}{}{#1}{#2}\right)\kern-.3em\right)}}

\newcommand{\qmah}[3] {\dist_{#1, #2}(#3)}

\def\syt{\operatorname{SYT}}

\def\Dinv{\operatorname{Dinv}}
\def\dinv{\operatorname{dinv}}

\def\shape{\operatorname{shape}}

\def\el{\ell}

\newtheorem{lemma}{Lemma}[subsection]
\newtheorem{prop}{Proposition}[subsection]
\newtheorem{cor}{Corollary}[subsection]
\newtheorem{thm}{Theorem}
\newtheorem{conj}{Conjecture}

\author[A.\ T.\ Wilson]{Andrew Timothy Wilson}
\title[Equidistribution on ordered multiset partitions]{An extension of MacMahon's Equidistribution Theorem to ordered multiset partitions}

\address{Department of Mathematics \\ University of Pennsylvania \\ 209 South 33rd Street \\ Philadelphia, PA 19104-6395}
\thanks{Partially supported by a National Defense Science and Engineering Graduate Fellowship and a National Science Foundation Mathematical Sciences Postdoctoral Research Fellowship. Much of this work was done while the author was a graduate student at UC San Diego.}
\email{atwilson0328@gmail.com}

\begin{document}
\begin{abstract}
A classical result of MacMahon states that inversion number and major index have the same distribution over permutations of a given multiset. In this work we prove a strengthening of this theorem originally conjectured by Haglund. Our result can be seen as an equidistribution theorem over the ordered partitions of a multiset into sets, which we call ordered multiset partitions. Our proof is bijective and involves a new generalization of Carlitz's insertion method. This generalization leads to a new extension of Macdonald polynomials for hook shapes. We use our main theorem to show that these polynomials are symmetric and we give their Schur expansion.
\end{abstract}

\maketitle

\tableofcontents

\section{Introduction}
\label{sec:intro}

\subsection{Permutation statistics}
\label{ssec:perm-stats}
Given a composition $\alpha$ of length $n$ (i.e.\ a vector of positive integers of length $n$), we let $\S_{\alpha}$ be the set of all permutations of the multiset $\{i^{\alpha_i} : 1 \leq i \leq n\}$. For a permutation $\sg \in \S_{\alpha}$, written in one-line notation, the \emph{descent} and \emph{ascent sets} of $\sigma$ are
\begin{equation*}
\begin{aligned}[c]
\Des(\sigma) = \{i  : \sigma_{i} > \sigma_{i+1} \}
\end{aligned}
\qquad
\begin{aligned}[c]
\Asc(\sigma) = \{i  : \sigma_{i} < \sigma_{i+1} \}
\end{aligned}
\end{equation*}
The \emph{inversions} of $\sigma$ are the pairs
\begin{align*}
\Inv(\sigma) &= \{(i, j) : 1 \leq i < j \leq n, \ \sigma_{i} > \sigma_{j} \} .
\end{align*}
It will be convenient to refine the set of inversions in the following manner:
\begin{equation*}
\begin{aligned}[c]
\Inv^{i, \Box} &= \{(i, j): i < j \leq n, \ \sigma_{i} > \sigma_{j} \} 
\end{aligned}
\qquad
\begin{aligned}[c]
\Inv^{\Box, j} = \{(i, j): 1 \leq i < j, \ \sigma_{i} > \sigma_{j} \} .
\end{aligned}
\end{equation*}
These are the elements of $\Inv(\sg)$ whose first (resp.\ second) coordinate is $i$ (resp.\ $j$). These sets allow us to define several statistics on $\S_{\alpha}$:
\begin{equation*}
\des(\sg) = |\Des(\sigma)| \quad \asc(\sigma) = |\Asc(\sigma)| \quad \inv(\sigma) = |\Inv(\sigma)| \quad  \maj(\sg) = \sum_{i \in \Des(\sg)} i .
\end{equation*}
These statistics are known as the \emph{descent number}, \emph{ascent number}, \emph{inversion number}, and \emph{major index} of $\sigma$, respectively. We will also make use of two refinements of inversion number:
\begin{align*}
 \inv^{i, \Box}(\sigma) = |\Inv^{i, \Box}(\sigma)| \qquad \inv^{\Box, j}(\sigma) = |\Inv^{\Box, j}(\sigma)| .
\end{align*} 

Given a statistic $\stat$ on $\S_\alpha$, the \emph{distribution} of $\stat$ over $\S_\alpha$ is the polynomial
\begin{align*}
\dists{\stat}{\alpha}{q} &= \sum_{\sigma \in \S_\alpha} q^{\stat(\sigma)} .
\end{align*}
When $\alpha = 1^n$, we will simply write $\dists{\stat}{n}{q}$. Two statistics, say $\stat$ on $\Obj$ and $\stat^{\prime}$ on $\Obj^{\prime}$, are said to be \emph{equidistributed} if their distributions are equal. 
One particularly nice way to prove equidistribution is to give a bijection $f : \Obj \to \Obj^{\prime}$ such that $\stat^{\prime}(f(\sg)) = \stat(\sg)$ for every $\sg \in \Obj$. Our main result will be a bijection of this form. Finally, it will be convenient to use interval notation for the integers, i.e.\ $[a,b]$ is the set of integers at least $a$ and at most $b$. 

In \cite{macmahon}, MacMahon showed that inversion number and major index are equidistributed over $\S_{\alpha}$, and that 
\begin{align*}
\dists{\inv}{\alpha}{q} = \dists{\maj}{\alpha}{q} = \qbinom{|\alpha|}{\alpha_{1}, \alpha_{2}, \ldots, \alpha_{n}} = \frac{\qint{|\alpha|}!}{\qint{\alpha_1}! \qint{\alpha_2}! \ldots \qint{\alpha_n}!}.
\end{align*}
where we use the standard $q$-analogs
\begin{align*}
\qint{n}! = \qint{n} \qint{n-1} \ldots \qint{1} \quad \quad \qint{n} = \frac{1-q^n}{1-q}.
\end{align*}
MacMahon's proof was not bijective; the first bijective proof of this fact was given in \cite{foata}. A second proof, essentially due to Carlitz \cite{carlitz}, is sometimes known as the \emph{insertion method}. It will be the template for our bijections in Section \ref{sec:eq-multiset}.

Part of our main result\footnote{To state our main result in full we will need more notation, which we will define later.} can be viewed as a bijective proof of the identity
\begin{align}
\label{main-result}
\sum_{\sg \in \S_{\alpha}} q^{\maj(\sg)} \prod_{i=1}^{\des(\sg)} \left( 1 + z / q^{i} \right) &= \sum_{\sg \in \S_{\alpha}} q^{\inv(\sg)} \prod_{j \in \Des(\sg)} \left( 1 + z/q^{\inv^{\Box, j} + 1 }\right)
\end{align}
for any composition $\alpha$. The case $\alpha = 1^n$ was proved in \cite{rw}. In order to provide our bijective proof, we will rephrase this statement as an equidistribution result on ordered multiset partitions. Then we will generalize Carlitz's insertion method to ordered multsiet partitions in order to complete our proof.

\subsection{The insertion method for $\S_n$}
\label{ssec:insertion-sn}
One consequence of MacMahon's equidistribution theorem is a pair of recursions for the distributions of the inversion number and the major index over $\S_{n}$:
\begin{equation}
\label{rec-sn}
\begin{aligned}[c]
\dists{\inv}{n}{q} = [n]_{q} \dists{\inv}{n-1}{q}
\end{aligned}
\qquad
\begin{aligned}[c]
\dists{\maj}{n}{q} = [n]_{q} \dists{\maj}{n-1}{q}.
\end{aligned}
\end{equation}
On the other hand, these two statements imply MacMahon's result. Carlitz's insertion method gives bijective proofs of these statements which can be combined to build a recursive bijection $\psi_{n}: \S_{n} \rightarrow \S_{n}$ such that $\maj(\psi(\sg)) = \inv(\sg)$. We say that $\psi_{n}$ maps the inversion number to the major index. We outline Carlitz's insertion method below.

To prove the left statement in \eqref{rec-sn}, one simply considers all the possible ways to insert $n$ into a permutation in $\S_{n-1}$ to create a permutation in $\S_{n}$. It is clear that, for $\sg \in \S_{n-1}$, inserting $n$ after the first $i$ elements of $\sg$ creates $n-i-1$ new inversions and does not affect the previously existing inversions. For example, for $\sg = 5167324 \in \S_{7}$, we can ``label'' these positions with subscripts that give the number of new inversions created by inserting an 8 at that position:
\begin{align*}
_{7}5_{6}1_{5}6_{4}7_{3}3_{2}2_{1}4_{0}.
\end{align*}
This proves the inversion side of \eqref{rec-sn}. The key to the insertion method is that something similar is true for the major index. In particular, we can label the spaces between elements of $\sg \in \S_{n-1}$, along with the left and right ends, according to the following scheme:
\begin{enumerate}
\item Label the position after $\sg_{n-1}$ with a zero.
\item Label the descents of $\sg$ right to left with $1, 2, \ldots, \des(\sg)$.
\item Label the position before $\sg_{1}$ with $\des(\sg)+1$.
\item Label the ascents of $\sg$ from left to right with $\des(\sg)+2, \ldots, n-1$.
\end{enumerate}
For example, $\sg = 5167324$ receives the following labels in this setting:
\begin{align*}
_{4}5_{3}1_{5}6_{6}7_{2}3_{1}2_{7}4_{0}
\end{align*}
These labels give the change in major index that comes from inserting $n$ at that position; one proof of this fact can be found in \cite{hlr}. This completes the proof of \eqref{rec-sn} and also gives a bijection that takes the inversion number to the major index. We include an example of this bijection in Figure \ref{fig:carlitz}. To compte $\psi_5(52143)$, we remove the $5$ and count the number of inversions lost by removing 5. In this case, we have lost 4 inversions. We record this number in the third column and the resulting permutation in the $\sg$ column. We repeat this process until we have reached $n = 1$ and filled the first three columns of the table. To build our new permutation, we recursively place $n$ at the position that receives label $i$ in the major index labeling. These labels have been italicized in the example.

\begin{figure}
\begin{center}
\begin{tabular}{l l l l }
$n$ & $\sg$ & Change in $\inv$  & $\psi_{n}(\sg)$ \\\hline
5 & 52143 &  & 24153 \\
4 & 2143 & 4 & $_{2}2_{3}4_{1}1_{\emph{4}}3_{0}$ \\
3 & 213 & 1 & $_{2}2_{\emph{1}}1_{3}3_{0}$ \\
2 & 21 & 0 & $_{2}2_{1}1_{\emph{0}}$ \\
1 & 1 & 1 & $_{\emph{1}}1_{0}$
\end{tabular}
\end{center}
 \caption{We compute $\psi_{5}(52143)$.}
\label{fig:carlitz}
\end{figure}

\subsection{The insertion method on $\S_{\alpha}$}
\label{ssec:insertion-sa}
It is natural to hope that this proof can be extended to permutations that may contain multiple copies of the same number. That is, we would like to give insertion proofs that
\begin{equation}
\label{rec-a}
\begin{aligned}[c]
\dists{\inv}{\alpha}{q} = \qbinom{|\alpha|}{\alpha_{n}} \dists{\inv}{\alpha^{-}}{q}
\end{aligned}
\qquad
\begin{aligned}[c]
\dists{\maj}{\alpha}{q} = \qbinom{|\alpha|}{\alpha_{n}} \dists{\maj}{\alpha^{-}}{q}.
\end{aligned}
\end{equation}
where $\alpha^{-} = (\alpha_1, \alpha_2, \ldots, \alpha_{n-1})$. Such proofs would imply MacMahon's equidistribution theorem and provide a bijection between inversion number and major index.

The inversion side cooperates nicely. As before, inserting an $n$ to the right of $i$ elements of $\sg \in \S_{\alpha}$ increases the inversion number by $|\alpha^{-}|-i$. Hence this insertion can create between 0 and $|\alpha^{-}|$ inversions. Furthermore, the position of a new $n$ has no affect on the number of inversions added by other $n$'s; in other words, each insertion is independent of the other insertions. This allows us to compute $\dists{\inv}{\alpha}{q}$ from $\dists{\inv}{\alpha^{-}}{q}$:
\begin{align*}
\dists{\inv}{\alpha}{q} &= \dists{\inv}{\alpha^{-}}{q} \left. \left(\prod_{i=0}^{|\alpha^{-}|} \frac{1}{1-q^{k}x}\right) \right|_{x^{\alpha_{n}}} = \qbinom{|\alpha|}{\alpha_{n}} \dists{\inv}{\alpha^{-}}{q} .
\end{align*}

To prove the major index side of \eqref{rec-a}, we essentially recreate the bijection constructed in \cite{foata-han, moon}. Let $\multiset{S}{k}$ denote the family of $k$-element multisets containing elements drawn from the set $S$. We would like to establish a bijection
\begin{align*}
\phi^{\maj}_{\alpha} : \S_{\alpha^{-}} \times \multiset{[0, |\alpha^{-}|]}{\alpha_{n}} \to \S_{\alpha}
\end{align*}
such that
\begin{align*}
\maj\left(\phi^{\maj}_{\alpha}(\sg, B)\right) = \maj(\sg) + \sum_{b \in B} b .
\end{align*}
Such a map would provide a combinatorial proof of the major index side of \eqref{rec-a}. Before inserting any $n$'s, we label $\sg \in \S_{\alpha^{-}}$ in a manner reminiscent of Section \ref{ssec:insertion-sn}:
\begin{enumerate}
\item Label the position after $\sg_{|\alpha^{-}|}$ with a zero.
\item Label the descents of $\sg$ right to left with $1, 2, \ldots, \des(\sg)$.
\item Label the position before $\sg_{1}$ with $\des(\sg)+1$.
\item Label the non-descents of $\sg$ from left to right with $\des(\sg)+2, \ldots, |\alpha^{-}|$.
\end{enumerate}
Write $B = \{b_{1} \geq b_{2} \geq \ldots \geq b_{\alpha_{n}}\}$. We insert an $n$ into the position labeled $b_{1}$. Then we go through the labeling process again, stopping once we have used the label $b_{1}$. We insert an $n$ into the position labeled $b_{2}$. We repeat this process until we have processed each element of $B$. We omit the proof that this map satisfies the desired properties, which can be found in \cite{foata-han, moon}. Instead, we will work through an example.

Let $\alpha = \{2, 1, 3, 2\}$, $\sg = 323113 \in \S_{\alpha^{-}}$, and $B = \{5^{2}\}$. We note that $\maj(\sg) = 4$. We begin by labeling $\sg$ according to the labeling associated with the major index.
\begin{align*}
_{3}3_{2}2_{4}3_{1}1_{5}1_{6}3_{0} .
\end{align*}
We place a 4 at the label 5 to get $3231413$. We relabel this permutation, stopping when we use the label 5. 
\begin{align*}
_{4}3_{3}2_{5}3_{2}1_{\ }4_{1}1_{\ }3_{0}
\end{align*}
Then we insert a 4 at the position labeled 5 to get $32431413$. As desired,
\begin{align*}
\maj(32431413) = 14 = \maj(\sg) + \sum_{b \in B} b = 4 + 5 + 5.
\end{align*}

Just as before, these insertion maps can be combined to yield a bijection $\psi_{\alpha} : \S_{\alpha} \to \S_{\alpha}$ that takes inversion number to major index. We illustrate $\psi_{\alpha}$ with the example in Figure \ref{fig:bij-a}. As in Section \ref{ssec:insertion-sn}, we fill the first three columns of the table from top to bottom by removing all copies of the largest element and recording the multiset of inversions lost during each removal, which we call $B$. Then we fill the fourth column by using the labeling associated with the major index to repeatedly insert a new element at the position that received the largest remaining label in $B$.

\begin{figure}
\begin{center}
\begin{align*}
\begin{array}{l l l l }
\alpha                          & \sg       & B  & \psi_{A}(\sg) \\\hline
\{1, 2, 1, 3\}      & 2443214   &                           & 4432124 \\
                            &           &                           & 4\ 4_{3}3_{2}2_{1}1\ 2_{\emph{0}} \\
                            &           &                           & 4_{\emph{3}}3_{2}2_{1}1\ 2_{0} \\
\{1, 2, 1\}             & 2321      & \{3, 3, 0\}               & _{\emph{3}}3_{2}2_{1}1_{4}2_{0} \\
\{1, 2\}                & 212       & \{2\}                     & _{\emph{2}}2_{1}1_{3}2_{0} \\
                            &           &                           & 2_{1}1_{\emph{0}} \\
\{1\}                & 1         & \{1, 0\}                  & _{\emph{1}}1_{0}  \\
\end{array}
\end{align*}
\end{center}
 \caption{An example of the map $\psi_{\alpha}$ for $\alpha = \{1,2,1,3\}$.}
\label{fig:bij-a}
\end{figure}

\section{Ordered Set and Multiset Partitions}
\label{sec:osp}

\subsection{Definitions}
The \emph{ordered set partitions} of order $n$ with $k$ blocks are partitions of the set $\{1,2,\ldots,n\}$ into $k$ subsets (called \emph{blocks}) with some order on the blocks. We write this set as $\osp{n}{k}$. For example, $13|45|2 \in \osp{5}{3}$, where we have listed each block as an increasing sequence and we have used bars to separate blocks. It is not difficult to see that $\osp{n}{n} = \S_n$, so ordered set partitions are a natural extension of permutations. 

More generally, given a composition $\alpha$ of length $n$, the ordered multiset partitions $\osp{\alpha}{k}$ are the partitions of the multiset $A(\alpha) = \{i^{\alpha_i}: 1 \leq i \leq n\}$ into $k$ ordered sets, which we still call blocks. For example, $24|134|2 \in \osp{(1,2,1,2)}{3}$. Note that, although we are dealing with the elements of a multiset, each block is still a set. The analogous objects where blocks are also multisets will not arise in our work.

So far, we have written each block of an ordered set or multiset partition in increasing order from left to right. We will often wish to use the opposite notation, i.e.\ we will write each block in decreasing order from left to right. Furthermore, we will use stars as subscripts to ``connect'' elements in the same block instead of bars to separate blocks. For example, the ordered multiset partition $24|134|2$ is written as $4_{\ast}2\ 4_{\ast} 3_{\ast} 1\ 2$ in this new notation. We will refer to an ordered multiset partition written this way as a \emph{descent-starred permutation}, since every permutation of the given multiset with some (but maybe not all) of its descents ``starred'' corresponds to an ordered multiset permutation in this fashion. More formally, we define the \emph{descent-starred permutations} of $A(\alpha)$ with $k$ stars as follows:
\begin{align*}
\S_{\alpha, k}^{>} = \{(\sg, S) : \sg \in \S_A, \,  S \subseteq \Des(\sg), \, |S| = k \} .
\end{align*}
The set $S$ corresponds to the entries of $\sg$ which are followed by stars. Then there is a straightforward bijection $\osp{\alpha}{k} \leftrightarrow \S_{\alpha,|\alpha|-k}^{>}$; given an ordered multiset partition, we write its blocks in decreasing order from left to right, add stars between adjacent elements that share a block, and remove the bars.

\subsection{Statistics}
We will study the following four statistics on ordered multiset partitions. All four of them appear in connection with a certain operator in the theory of diagonal harmonics which we call the Garsia-Haiman delta operator. Two of them ($\inv$ and $\maj$) are directly involved with the statement \eqref{main-result} given in the introduction.

First, given $\pi \in \osp{\alpha}{k}$, $\inv(\pi)$ counts the number of pairs $a > b$ such that $a$'s block is strictly to the left of $b$'s block in $\pi$ and $b$ is minimal in its block in $\pi$. We call these pairs \emph{inversions}. For example, $15|23|4$ has two inversions, between the 5 and the 2 and the 5 and the 4. 

For any $\pi \in \osp{\alpha}{k}$, we number $\pi$'s blocks $\pi_1, \pi_2, \ldots, \pi_k$ from left to right. Let $\pi^h_i$ by the $h$th smallest element in $\pi_i$, beginning at $h=1$. Then the \emph{diagonal inversions} of $\pi$, written $\Dinv(\pi)$, are the triples
\begin{align*}
\{(h, i, j) : 1 \leq i < j \leq k, \ \pi^h_i > \pi^h_j \} \cup \{(h, i, j): 1 \leq i < j \leq k, \ \pi^h_i < \pi^{h+1}_j \} .
\end{align*}
The triples of the first type are \emph{primary diagonal inversions}, and the triples of the second type are \emph{secondary diagonal inversions}. We set $\dinv(\pi)$ to be the cardinality of $\Dinv(\pi)$. For example, consider the ordered multiset permutation $24|134|2$. It is helpful to ``stack'' the elements in each block vertically, obtaining the diagram
\begin{align*}
\begin{ytableau}
\none & 4 & \none \\
4 & 3 & \none \\
2 & 1 & 2
\end{ytableau}
\end{align*}
The primary diagonal inversions are $(1,1,2)$ (between the leftmost 2 and the 1 in the first row) and $(2,1,2)$ (between the 4 and the 3 in the second row) and the only secondary diagonal inversion is $(1,1,2)$ (between the leftmost 2 in the first row and the 3 in the second row), for a total of three diagonal inversions.

To define the \emph{major index} of $\pi$, we consider the permutation $\sg = \sg(\pi)$ obtained by writing each block of $\pi$ in decreasing order. Then we recursively form a word $w$ by setting $w_0=0$ and $w_i = w_{i-1} + \chi(\sg_i \text{ is minimal in its block in } \pi)$ for each $i > 0$. Then we set
\begin{align*}
\maj(\pi) &= \sum_{i : \ \sg_i > \sg_{i+1} } w_i .
\end{align*}
Using the ordered multiset permutation $\pi = 24|134|2$ again, we obtain $\sg = 424312$ and $w = 0011123$, beginning with $w_0=0$. The descents of $\sg$ occur at positions 1,3, and 4, so $\maj(\pi) = w_1 + w_3 + w_4 = 0 + 1 + 1 = 2$. 

There is an alternate definition of the major index which we will use in some of our proofs. It is clear from the definition above that for any $\pi \in \osp{n}{n}$, the major index defined here is equivalent to the major index defined on permutations in Section \ref{sec:intro}. Now we consider what happens to $\maj(\pi)$ if we decide to insert a star after a descent at position $d$. We note that $w_i$ decreases by 1 for each $i \geq d$. Therefore the major index of $\pi$ has decreased by 1 for each descent weakly to the right of position $d$. Therefore, if $(\sg, S)$ is the descent-starred permutation representation of $\pi$, we can write
\begin{align}
\label{maj-alt}
\maj(\pi) = \maj(\sg) - \sum_{i \in S} |\Des(\sg) \cap \{i,i+1,\ldots\}| .
\end{align}
This shows that
\begin{align*}
\distosp{\maj}{\alpha}{k}{q} &= \left. \prod_{\sg \in \S_{\alpha}} q^{\maj(\sg)} \prod_{i=1}^{\des(\sg)} \left(1 + z / q^{i} \right) \right|_{z^{|\alpha| - k}} .
\end{align*}

Finally, we define the \emph{minimum major index} of $\pi$ as follows. We begin by writing the elements of $\pi_k$ in increasing order from left to right. Then, recursively for $i = k-1$ to $1$, we choose $r$ to be the largest element in $\pi_i$ that is less than or equal to the leftmost element in $\pi_{i+1}$, as previously recorded. If there is no such $r$, we write $\pi_i$ in increasing order. If there is such an $r$, beginning with $\pi_i$ in increasing order, we cycle its elements until $r$ is the rightmost element in $\pi_i$. We write down $\pi_i$ in this order. We continue this process until we have processed each block of $\pi$. For example, consider the ordered multiset permutation $\pi = 13|23|14|234$. Processing the blocks of $\pi$ from right to left, we obtain $312341234$.  We consider the result as a permutation, which we denote $\tau = \tau(\pi)$, and define
\begin{align*}
\minimaj(\pi) = \sum_{i: \ \tau_i  > \tau_{i+1}} i
\end{align*}
i.e.\ the major index of the permutation $\tau$. The name $\minimaj$ comes from the fact that $\minimaj(\pi)$ is equal to the minimum major index achieved by any permutation that can be obtained by permuting elements within the blocks of $\pi$. 

\section{Equidistribution on Ordered Multiset Partitions}
\label{sec:eq-multiset}

In this section, we prove the following equidistribution theorem for ordered multiset partitions. 

\begin{thm}
\label{thm:eq-multiset}
For any composition $\alpha$,
\begin{align*}
\distosp{\inv}{\alpha}{k}{q} &= \distosp{\maj}{\alpha}{k}{q} = \distosp{\dinv}{\alpha}{k}{q}.
\end{align*}
\end{thm}

The reader may have noticed that $\distosp{\minimaj}{n}{k}{q}$ is not included in the list of equidistributed polynomials; that is because we have not proved this case at this point. We describe the unique difficulties of this case in Subsection \ref{ssec:minimaj}. 

Our proof of Theorem \ref{thm:eq-multiset} is bijective and employs a generalization of Carlitz's insertion method from permutations to ordered multiset partitions. We describe ``insertion maps'' for $\inv$, $\dinv$, and $\maj$ in Subsections \ref{ssec:omp-inv},  \ref{ssec:omp-maj}, and \ref{ssec:omp-dinv}, respectively. These maps will be of the form 
\begin{align*}
\phi^{\inv}_{\alpha,k,\ell} &: \osp{\alpha^{-}}{\ell} \times \binom{[0, \ell-1]}{\alpha_n - k + \ell} \times \multiset{[0,\ell]}{ k - \ell} \to \osp{\alpha}{k} \\
\phi^{\maj}_{\alpha,k,\ell} &: \osp{\alpha^{-}}{\ell} \times \binom{[0, \ell-1]}{\alpha_n - k + \ell} \times \multiset{[0,\ell]}{ k - \ell} \to \osp{\alpha}{k} \\
\phi^{\dinv}_{\alpha,k,\ell} &: \osp{\alpha^{-}}{\ell} \times \binom{[0, \ell-1]}{\alpha_n - k + \ell} \times \multiset{[0,\ell]}{ k - \ell} \to \osp{\alpha}{k} 
\end{align*}
where we have used some new notation here. For any composition $\alpha$ of length $n$, $\alpha^{-}$ is the composition obtained by removing the rightmost entry in $\alpha$. For a set $S$, $\binom{S}{k}$ is defined to be the collection of all subsets of $S$ of size $k$. Similarly, $\multiset{S}{k}$ is the collection of all multisets of size $k$ whose elements are taken (possibly more than once) from $S$. By ``insertion maps,'' we mean that they satisfy the properties
\begin{align*}
\inv \left(\phi^{\inv}_{\alpha, k, \ell} (\pi, U, B) \right) &= \inv (\pi) + \sum_{u \in U} u + \sum_{b \in B} b \\
\maj \left(\phi^{\maj}_{\alpha, k, \ell} (\pi, U, B) \right) &= \maj (\pi) + \sum_{u \in U} u + \sum_{b \in B} b \\
\dinv \left(\phi^{\dinv}_{\alpha, k, \ell} (\pi, U, B) \right) &= \dinv (\pi) + \sum_{u \in U} u + \sum_{b \in B} b .
\end{align*}
When $k = |\alpha|$ these maps will reduce to the insertion processes defined in Subsection \ref{ssec:insertion-sa}. We will use these maps to construct a bijective proof of Theorem \ref{thm:eq-multiset}. Subsection \ref{ssec:omp-dist} contains more information about the shared distribution of the polynomials in Theorem \ref{thm:eq-multiset}.

Finally, it is natural to wonder if our results can be transferred to ordered partitions of a multiset into \emph{multisets} instead of sets. For example, $113|23|1$ is one such object. At this point, we have not managed to accomplish this task; in particular, the major index statistics seems to behave differently in this setting. However,

\subsection{Insertion for $\inv$}
\label{ssec:omp-inv}
Recall that we need to define a map of the form
\begin{align*}
\phi^{\inv}_{\alpha,k,\ell} &: \osp{\alpha^{-}}{\ell} \times \binom{[0, \ell-1]}{\alpha_n - k + \ell} \times \multiset{[0,\ell]}{ k - \ell} \to \osp{\alpha}{k} .
\end{align*}
We can think of the set $U \in \binom{[0, \ell-1]}{\alpha_n - k + \ell}  $ as providing the increases in $\inv$ that come from adding a new $n$ without creating a new block, and the multiset $B \in \multiset{[0,\ell]}{ k - \ell}$  as providing the increases in the statistic that come from adding a new $n$ while creating a new block. These maps will cooperate with our inversion statistic in the following manner:
\begin{align}
\label{inv-prop}
\inv \left(\phi^{\inv}_{\alpha, k, \ell} (\pi, U, B) \right) &= \inv (\pi) + \sum_{u \in U} u + \sum_{b \in B} b
\end{align}

Given $\pi \in \osp{\alpha^{-}}{\ell}$, we label each block of $\pi$ from right to left with the numbers $0, 1, \ldots, \ell-1$. We repeatedly remove the largest element from the multiset $U \cup B$, taking the element from $U$ if the largest elements are equal. Call this element $i$. If $i$ came from $U$, we place an $n$ in the block that received the label $i$. If $i$ came from $B$ and is equal to $\ell$, we place an $n$ as a new block to the left of the block that received the label $i$. If $i$ came from $B$ and is less than $\ell$, we place an $n$ as a new block just to the right of the block labeled $i$. The resulting ordered multiset partition is $\phi^{\inv}(\pi, U, B)$. 

For example, say $\alpha = \{1, 2, 2, 3\}$, $k=6$, and $\ell = 5$. We consider $\pi = 3 | 1 | 2 | 2 | 1 3 \in \osp{\alpha^{-}}{5}$, $U = \{2, 0\}$, and $B = \{3\}$. Then $U \cup B = \{3, 2, 0\}$. We label $\pi$ as follows, with the labels written as subscripts at the end of each block:
\begin{align*}
3_4 | 1_3 | 2_2 | 2_1 | 1 3_0 \\
\end{align*}
The largest element in $U \cup B$ is 3 and it comes from $B$, so we insert a 4 at the position labeled 3 as a new block to the right of the block labeled with the 3.
\begin{align*}
3_4 | 1_3 | 4 | 2_2 | 2_1 |  1 3_0
\end{align*}
Now the largest remaining element of $U \cup B$ is 2 and it comes from $U$, so we put a 4 into the block labeled 2. 
\begin{align*}
3_4 | 1_3 | 4 | 2 4_2 | 2_1 | 1 3_0
\end{align*}
Finally, we insert a 4 into the block labeled 0 to obtain $\phi^{\inv}(\pi, U, B)$.
\begin{align*}
3_4 | 1_3 | 4 | 2 4_2 | 2_1 | 1 3 4_0
\end{align*}
We can check that \eqref{inv-prop} holds here.
\begin{align*}
10 &= \inv \left( 3 | 1 | 4 | 4 2 | 2 | 4 3 1 \right) \\
 	&=\inv \left( 3 | 1 | 2 | 2 | 3 1 \right) + \sum_{u \in U} u + \sum_{b \in B} b \\
	&= 5 + (2 + 0 )+ 3 .
\end{align*}

\begin{lemma}[Insertion for $\inv$]
\label{lemma:inv-insertion}
For any composition $\alpha$ of length $n$ and positive integers $\ell \leq k$, $\phi^{\inv}_{\alpha,k,\ell}$ is well-defined and injective. The image of $\phi^{\inv}_{\alpha,k,\ell}$ is the ordered multiset partitions $\pi \in \osp{\alpha}{k}$ with exactly $k - \ell$ singleton blocks containing $n$. Furthermore, for any $\pi \in \osp{\alpha}{\ell}$, $U \in \binom{[0, \ell-1]}{\alpha_n - k + \ell}$, and $B \in \multiset{[0,\ell]}{k - \ell}$, 
\begin{align*}
\inv \left(\phi^{\inv}_{\alpha, k, \ell} (\pi, U, B) \right) &= \inv (\pi) + \sum_{u \in U} u + \sum_{b \in B} b .
\end{align*}

\end{lemma}

\begin{proof}
The only way $\phi^{\inv}_{\alpha,k,\ell}$ could not be well-defined is if we tried to insert two $n$'s into the same block. Since $U$ is a set, this does not occur. The statement about the image of $\phi^{\inv}_{\alpha,k,\ell}$ follows from the definition. Furthermore, the insertion map is clearly bijective, and any function that bijects onto its image is an injection.

We will prove that whenever we remove the largest element $i$ from $U \cup B$ and introduce a new $n$ to our ordered multiset partition as described in the insertion map, we introduce $i$ new inversions to the ordered multiset partition. Specifically, we create inversions between this new $n$ and the minimal elements of the $i$ labeled blocks of $\pi$ that are to the right of the block that received label $i$. Since $n$ is the largest entry, we do not create any new inversions that end at $n$. Finally, we note that we do not destroy any inversions that existed before we inserted this new $n$. 
\end{proof}

\subsection{Insertion for $\maj$}
\label{ssec:omp-maj}
To define $\phi^{\op{\maj}}_{\alpha, k, \ell}$, we will view $\pi \in \osp{\alpha^{-}}{\ell}$ as a descent-starred multiset permutation $(\sg, S)$. We will label the unstarred positions of $\sg$ as in Section \ref{sec:intro}. Specifically, we label the unstarred descents from right to left, then the non-descents from left to right, using increasing labels $0, 1, \ldots, \ell$. Let $U^{+} = \{u+1 : u \in U\}$. We repeatedly remove the largest element $i$ from $U^{+} \cup B$, taking $i$ from $B$ if the largest elements are equal. Then we proceed through the following algorithm:
\begin{enumerate}
\item Insert an $n$ at the position labeled $i$.
\item Move each star to the right of the new $n$ one descent to the left.
\item If $i$ came from $U^{+}$, star the rightmost descent.
\item Relabel as before, stopping at the label $i$ if $i$ came from $B$ and $i-1$ if $i$ came from $U^{+}$.
\end{enumerate}
When we have used each element of $U^{+} \cup B$, the result is $\phi^{\maj}_{\alpha, k, \ell}((\sg, S), U, B)$. 

For example, let us again consider $\alpha = \{1, 2, 2, 3\}$, $k=6$, and $\ell = 5$ with $(\sg, S) = 3\ 1\ 2\ 2\ 3_{*}1 \in \osp{\alpha^{-}}{5}$, $U = \{2, 0\}$, and $B = \{3\}$. Then $U^{+} \cup B = \{3,3,1\}$. We label $(\sg, S)$ as follows.
\begin{align*}
_{2}3_{1}1_{3}2_{4}2_{5}3_{*}1_{0} .
\end{align*}
We insert a 4 at the position labeled 3 and then move all stars to the right of that position one spot to their left. Since we took 3 from $B$, we do not star the rightmost descent, resulting in $3\ 1\ 4_{*}2\ 2\ 3\ 1$. We re-label to obtain the following.
\begin{align*}
_{3}3_{2}1\ 4_{*}2\ 2\ 3_{1}1_{0}
\end{align*}
Again we choose the position labeled 3. This time we star the rightmost descent after shifting stars because 3 came from $U^{+}$, yielding $4\ 3_{*}1\ 4\ 2\ 2\ 3_{*}1$. Finally we label this element
\begin{align*}
4_{2}3_{*}1\ 4_{1}2\ 2\ 3_{*}1_{0}.
\end{align*}
We insert a 4 at the position labeled 1, shift stars, and star the rightmost descent to get $4\ 3_{*}1\ 4\ 4_{*}2\ 2\ 3_{*}1$. We check that this new permutation has the desired major index.
\begin{align*}
10 &= \maj \left( 4\ 3_{*}1\ 4\ 4_{*}2\ 2\ 3_{*}1 \right) = 1 + 1 + 3 + 5  \\
	&= \maj \left(3\ 1\ 2\ 2\ 3_{*}1\right) + \sum_{u \in U} u + \sum_{b \in B} b \\
	&= (1 + 4) + (2 + 0) + 3 .
\end{align*} 
Now we prove that this process satisfies the necessary properties.

\begin{lemma}[Insertion Lemma for $\maj$]
\label{lemma:maj-insertion}
For any composition $\alpha$ of length $n$ and positive integers $\ell \leq k$, $\phi^{\maj}_{\alpha,k,\ell}$ is well-defined and injective. Furthermore, for any $\pi \in \osp{\alpha^{-}}{\ell}$, $U \in \binom{[0, \ell-1]}{\alpha_n -k + \ell}$, and $B \in \multiset{[0,\ell]}{k - \ell}$, 
\begin{align*}
\op{\maj} \left(\phi^{\op{\maj}}_{\alpha, k, \ell} (\pi, U, B) \right) &= \op{\maj} (\pi) + \sum_{u \in U} u + \sum_{b \in B} b .
\end{align*}

\end{lemma}

\begin{proof}
To show that $\phi^{\maj}_{\alpha,k,\ell}$ is injective, we describe its inverse. Given some descent-starred multiset permutation $(\tau, T)$, we first check if the rightmost descent of $(\tau, T)$ is starred. If it is, we remove the star and prepare to add an element to $U$. Otherwise, we prepare to add an element to $B$. We scan $\tau$ for the rightmost $n$ which is either at the right end of $\tau$ or between two entries such that the entry to the left of $n$ is greater than the entry to the right of $n$. If there is no such $n$, we choose the leftmost $n$ in $\tau$. We move all stars that are weakly to the right of this $n$'s position one descent to their right and then remove $n$. Say that, at this point, we have decreased the original major index of $(\tau, T)$ by $i$. We add $i$ to either $U$ or $B$, as decided above. Then we repeat this process until we have removed all $n$'s. It is important to note that the inverse does not depend on knowledge of $\ell$; therefore, each $\rho \in \osp{\alpha}{k}$ is in the image of $\phi^{\maj}_{\alpha, k, \ell}$ for a unique value of $\ell$. 

 In the remainder of the proof, we show that the map cooperates with the statistic $\maj$ as proposed in the lemma. We consider the labeling of the descent-starred multiset permutation $(\sg,S)$ equivalent to $\pi$ at any step during the insertion process. Say $i$ is currently the largest element of $U^{+} \cup B$.

Assume first that the position labeled $i$ is a descent. We use $d$ to denote the number of starred descents to the right of this position. By the insertion method discussed in Subsection \ref{ssec:insertion-sa}, inserting $n$ into the position labeled $i$ increases $\maj(\sg)$ by $i+d$. Furthermore, we have not created a new descent, so the number of descents weakly to the right of any starred position has remained the same. Therefore, by the alternate definition of $\maj(\pi)$ in \eqref{maj-alt}, we have increased $\maj(\pi)$ by $i+d$ after Step 1. 

For Step 2, we move all stars to the right of the position labeled $i$ one descent to their left. Since position $i$ contains an unstarred descent, this is always possible. Furthermore, each of these $d$ stars have picked up an additional descent that is weakly to their right. Using \eqref{maj-alt} again, we see that the change in $\maj(\pi)$ after Step 2 is $i+d-d=i$. 

Finally, we need to consider if $i$ came from $U^{+}$ or $B$. If $i$ came from $U^{+}$, we star the rightmost descent. This subtracts 1 from $\maj(\pi)$. In either case, we have increased the major index of $\pi$ by the amount equal to the element from $U$ or $B$ corresponding to $i$. 

By the insertion process from Subsection \ref{ssec:insertion-sa}, we can relabel the resulting descent-starred permutation and repeat the process as long as we bound the labels as described in Step 4. Then, by the same argument as we used above, the insertion process will modify the major index as described in the statement of the lemma.

Now we consider where the argument must change when the position labeled $i$ is not a descent. We still use $d$ to denote the number of starred descents to the right of position $i$, and we set $c$ to be the number of starred descent to this position's left. Since every starred descent occurs before the position labeled $i$ in the labeling order for $\maj(\sg)$, Step 1 increases $\maj(\pi)$ by $i+c+d$. For Step 2, the position labeled $i$ still contains an unstarred descent, so we can still move the stars as described. As before, this means that each of the $d$ stars to the right of the position labeled $i$ adds a descent to its right, contributing $-d$ to $\maj(\pi)$. Furthermore, inserting $n$ at a non-descent creates a new descent, so each of the $c$ stars to the left of the position labeled $i$ has added a descent to its right, contributing $-c$ to $\maj(\pi)$. Therefore the total increase of $\maj(\pi)$ is $i+c+d-c-d=i$. Steps 3 and 4 do not depend on whether we are inserting at a descent or a non-descent.
\end{proof}

These two insertion maps work together to provide a bijection $\psi_{\alpha, k}: \osp{\alpha}{k} \to \osp{\alpha}{k}$ that takes inversion number to major index. The bijection is described recursively as follows.
\begin{enumerate}
\item Given an ordered multiset partition $\rho \in \osp{\alpha}{k}$, choose $\ell$ such that $\rho$ has $k - \ell$ singleton blocks containing $n$. 
\item Set $(\pi, U, B)$ to be the inverse of $\rho$ under $\phi^{\inv}_{\alpha, k, \ell}$. 
\item Recursively send $\pi$ to $\pi^{\prime} = \psi_{\alpha^{-},\ell}(\pi)$. 
\item Set $\psi_{\alpha, k}(\rho) = \phi^{\maj}_{\alpha, k, \ell} (\pi^{\prime}, U, B)$. 
\end{enumerate}
Finally, in order to begin the recursion, we declare that $\psi_{1^m,m}$ is the identity map. We work through an example of this bijection in Figure \ref{fig:gen-bij}.

\begin{prop}
For any composition $\alpha$ and positive integer $k$, $\psi_{\alpha,k}$ is a bijection with the property
\begin{align*}
\maj(\psi_{\alpha,k}(\rho)) = \inv(\rho)
\end{align*}
for any $\rho \in \osp{\alpha}{k}$. 
\end{prop}

\begin{proof}
We will work by induction on $n$, the length of $\alpha$. If $n = 1$, Then the multiset is $\{1^{\alpha_1}\}$ and $k$ must be equal to $\alpha_1$. In this case $\osp{\alpha}{k}$ only has 1 element, which has $\alpha_1$ parts all equal to 1.  This element clearly has $\inv = \maj = 0$. We defined $\psi_{\alpha,k}$ so that it is the identity in this case, which clearly is a bijection and satisfies $\maj(\psi_{\alpha,k}(\pi)) = \inv(\pi)$ for the unique $\pi \in \osp{\alpha}{k}$. 

If $n > 1$, take any element $\rho \in \osp{\alpha}{k}$. We choose $\ell$, $\pi$, $\pi^{\prime}$, $U$, and $B$ as instructed in the definition of $\psi_{\alpha,k}$. The images of $\phi^{\inv}_{\alpha, k, \ell}$ for $\ell = 1$ to $k$ partition $\osp{\alpha}{k}$ into the subsets consisting of elements which have $k-\ell$ singleton $n$ blocks. (Since we assume each $\alpha_i > 0$, an element cannot consist entirely of singleton $n$ blocks.) We also noted while proving Lemma \ref{lemma:maj-insertion} that each $\rho \in \osp{\alpha}{k}$ is in the image of $\phi^{\maj}_{\alpha,k,\ell}$ for a unique value of $\ell$. Since each of these insertion maps is invertible, $\psi_{\alpha,k}$ is a bijection.

Finally, we use Lemmas \ref{lemma:inv-insertion} and \ref{lemma:maj-insertion} along with the inductive hypothesis to compute
\begin{align*}
\maj(\psi_{\alpha,k}(\rho)) &=  \maj(\pi^{\prime}) + \sum_{u \in U} u + \sum_{b \in B} b \\
	&= \inv(\pi) + \sum_{u \in U} u + \sum_{b \in B} b \\
	&= \inv(\rho) \qedhere. 
\end{align*}
\end{proof}

We work through an example of the map $\psi_{A,k}(\pi)$ in Figure \ref{fig:gen-bij}. We repeatedly remove all of the largest elements (and their stars) from the starred permutation and recording the number of inversions lost in the $U$ and $B$ columns. Starred elements contribute to the $U$ column and unstarred elements contribute to the $B$ column. Once we have reached the final row, we use this information to build the $\psi_{\alpha, k}((\sg, S))$ column from bottom to top. We use the elements of $U^{+} \cup M$ to select the positions at which to insert new largest elements. This insertion follows the procedure laid out in the definition of $\phi^{\op{\maj}}$.

\begin{figure}
\begin{center}
\begin{align*}
\begin{array}{l l l l l l}
\alpha			& k 	& \pi				& U	 		& B  		& \psi_{A, k}(\pi) \\\hline
\{3, 1, 2, 1\} 	& 3 	& 134 | 1 | 3 | 12 		& 			& 			& 3_{*}2_{*}1\ 1\ 4\ 3\ 1\\
\{3, 1, 2\}  	& 2	& 13|1|3|12		& \{4\}		& \emptyset	& _{2}3_{*}2_{*}1_{3}1_{4}3_{1}1_{0}\\
				&	&						&			&			& _{\emph{1}}2_{*}1_{2}1\ 3_{*}1_{0}   \\
\{3, 1\} 		& 1	& 1 | 1 | 12				& \{2\}		& \{1\}		& _{1}2_{*}1_{2}1_{\emph{3}}1_{0} \\
\{3\}  			& 0 	& 1|1|1 					& \{0\}		& \emptyset	& _{\emph{1}}1_{2}1_{3}1_{0} \\
\end{array}
\end{align*}
\end{center}
 \caption{
An example of the map $\psi_{A,k}(\pi)$.
}
\label{fig:gen-bij}
\end{figure}

There are a number of other consequences of our proof. For example, the \emph{right-to-left minima} of a permutation $\sg$ are the entries $\sg_i$ such that, for all $j>i$, $\sg_i < \sg_j$. In \cite{rw}, Remmel and the author proved that the $\alpha = 1^n$ case of $\psi_{\alpha,k}$ preserves the right-to-left minima of $\sg$, where $(\sg, S)$ is the descent-starred permutation representation of an ordered partition of $\alpha$.  The same is true for general $\psi_{\alpha,k}$.

\begin{cor}
\label{cor:rlmin}
For any element of $\pi \in \osp{\alpha}{k}$ considered as a descent-starred permutation $(\sg,S)$, consider its image $\psi_{\alpha,k}(\pi)$ as the descent-starred permutation $(\tau, T)$. Then $\sg$ and $\tau$ have the same right-to-left minima. In particular, $\sg_n = \tau_n$.
\end{cor}

The crux of the proof is that the insertion algorithms only change the last element of $\sg$ when 0 is an element of the multiset $B$.

\subsection{Insertion for $\dinv$}
\label{ssec:omp-dinv}

In order to prove that $\dinv$ is equidistributed with $\inv$ and $\maj$, we define an insertion map for $\dinv$
\begin{align*}
\phi^{\dinv}_{\alpha,k,\ell} &: \osp{\alpha^{-}}{\ell} \times \binom{[0, \ell-1]}{\alpha_n - k + \ell} \times \multiset{[0,\ell]}{ k - \ell} \to \osp{\alpha}{k} .
\end{align*}
Given $\pi \in \osp{\alpha^{-}}{\ell}$, we will use two different labelings to insert the $n$'s into $\pi$. We label the positions between the blocks, as well as the positions at either end of $\pi$, with the labels $0, 1, \ldots, \ell$ from right to left. We will call these the \emph{gap labels}. 

We will label the $\ell$ blocks of $\pi$ with the labels $0, 1, \ldots, \ell-1$. Set $h$ to be the maximum height of any element in $\pi$. We obtain the \emph{block labels} of $\pi$ by, beginning with the label 0, repeatedly
\begin{enumerate}
\item labeling the elements of $\pi$ at height $h$ from left to right with increasing labels, and
\item decrementing $h$.
\end{enumerate}
We repeat until $h < 0$. For example, if $\pi = 124|2|13|134|1$, the block labels of $\pi$ are $0|3|2|1|4$.

With these labels in hand, we define $\phi^{\dinv}_{\alpha,k,\ell}(\pi, U, B)$ by inserting an $n$ into each block that receives a block label $u \in U$ and into each gap that receives a gap label $b \in B$.  As usual, the key is to prove that this map cooperates with the statistic $\dinv$.

\begin{lemma}[Insertion for $\dinv$]
\label{lemma:dinv-insertion}
For any composition $\alpha$ of length $n$ and positive integers $\ell \leq k$, $\phi^{\dinv}_{\alpha,k,\ell}$ is well-defined and injective. The image of $\phi^{\dinv}_{\alpha,k,\ell}$ is the ordered multiset partitions $\pi \in \osp{\alpha}{k}$ with exactly $k - \ell$ singleton blocks containing $n$. Furthermore, for any $\pi \in \osp{\alpha}{\ell}$, $U \in \binom{[0, \ell-1]}{\alpha_n - k + \ell}$, and $B \in \multiset{[0,\ell]}{k - \ell}$, 
\begin{align*}
\op{\dinv} \left(\phi^{\op{\dinv}}_{\alpha, k, \ell} (\pi, U, B) \right) &= \op{\dinv} (\pi) + \sum_{u \in U} u + \sum_{b \in B} b .
\end{align*}
\end{lemma}

\begin{proof}
It is clear that inserting an $n$ at a gap labeled $b$ creates $b$ new diagonal inversions, one with each block to the right of the gap, and does not affect any other diagonal inversions. 

Now say we insert an $n$ into a block $\pi_i$ labeled $u$. Say that $\pi_i$ had size $s$ before we added an $n$. After inserting an $n$, it has size $s+1$ with an $n$ at height $s+1$. We claim that we have created one new diagonal inversion for each block of size greater than $s$ and for each block of size $s$ that is to the left of $\pi_i$. First, consider a block $\pi_j$ with $|\pi_j| > s$. If $j > i$, then $(s+1,i,j)$ is a new primary diagonal inversion; if $j < i$, then $(s,j,i)$ is a new secondary diagonal inversion. There are no other new diagonal inversions between $\pi_i$ and $\pi_j$. Now we consider $\pi_j$ with $j < i$ and $|\pi_j| = s$. There is a new secondary diagonal inversion $(s,j,i)$. By the definition of our insertion map, there are exactly $u$ such blocks, so we have created $u$ new diagonal inversions. 

Finally, we note that, since two $n$'s cannot form a diagonal inversion, each insertion of an $n$ does not affect previous or subsequent insertions. 
\end{proof}

It follows from Lemma \ref{lemma:dinv-insertion} that $\distosp{\dinv}{\alpha}{k}{q} = \distosp{\inv}{\alpha}{k}{q} = \distosp{\maj}{\alpha}{k}{q}$. We can form a bijection between any pair of these statistics by using the definition of $\psi_{\alpha,k}$ as a template. Furthermore, any such bijection preserves right-to-left minima.  

\subsection{The statistic $\minimaj$}
\label{ssec:minimaj}

Data computed in Sage suggests that $\minimaj$ shares the distribution of $\inv$, $\maj$, and $\dinv$; unfortunately, we cannot prove this with the techniques currently available to us. However, data yields the following conjectures as to how this statistic relates to the previous statistics. Given an ordered set partition $\pi \in \ospn{n}$, we say the \emph{shape} of $\pi$, written $\shape(\pi)$, is the composition whose $i$th block is equal to the size of $\pi^{i}$, the $i$th block in $\pi$ from left to right.

\begin{conj}
\label{conj:minimaj}
For any composition $\beta$ of length $n$,
\begin{align*}
\sum_{\stackrel{\pi \in \ospn{n}}{\shape(\pi) = \beta}} q^{\inv(\pi)} &= \sum_{\stackrel{\pi \in \ospn{n}}{\shape(\pi) = \beta}} q^{\minimaj(\pi)} .
\end{align*}
Also, for any composition $\alpha$
\begin{align*}
\distosp{\minimaj}{\alpha}{k}{q} &= \distosp{\inv}{\alpha}{k}{q} = \distosp{\maj}{\alpha}{k}{q} = \distosp{\dinv}{\alpha}{k}{q}.
\end{align*}
Note that neither statement directly implies the other.
\end{conj}

In personal communication, Brendon Rhoades has notified the author that he is currently preparing a proof of both parts of this conjecture. 

\subsection{The Mahonian distribution on $\osp{\alpha}{k}$}
\label{ssec:omp-dist}

In this subsection, we describe the distribution shared by the statistics $\inv$, $\maj$, and $\dinv$ on $\osp{\alpha}{k}$. Define the \emph{Mahonian distribution} on $\osp{\alpha}{k}$ to be the polynomial
\begin{align*}
\qmah{\alpha}{k}{q} &= \distosp{\inv}{\alpha}{k}{q} = \distosp{\maj}{\alpha}{k}{q} = \distosp{\dinv}{\alpha}{k}{q} .
\end{align*}
We know from MacMahon's theorem that $\qmah{\alpha}{|\alpha|}{q}= \qbinom{|\alpha|}{\alpha_{1}, \ldots, \alpha_{n}}$. In general, we can only give a recursive description of $\qmah{\alpha}{k}{q}$. Applying standard $q$-binomial identities to the insertion maps given above, we see
\begin{align}
\label{dist-rec}
\qmah{\alpha}{k}{q} &= \sum_{l=1}^{k} q^{\binom{\alpha_n-k+\el}{2}} \qbinom{\el}{\alpha_n-k+\el} \qbinom{k}{\el} \qmah{\alpha^{-}}{\el}{q}
\end{align}
with initial condition
\begin{align}
\label{initial-condition}
\qmah{(\alpha_1)}{k}{q} &= \chi(k = \alpha_1).
\end{align} 
We can cancel terms of (\ref{dist-rec}) to obtain the identity
\begin{align}
\label{mah-rec}
\qmah{\alpha}{k}{q} &= \sum_{l=1}^{k} q^{\binom{\alpha_n-k+\el}{2}} \qbinom{k}{\alpha_n-k+\el,k-\alpha_n,k-\el} \qmah{\alpha^{-}}{\el}{q} .
\end{align}

We can obtain another expression for this polynomial in the special case $\alpha_{1} = \ldots = \alpha_{n} = a$. Before we can state this expression, we must define a $q$-analog of the (generalized) Stirling numbers of the second kind. The $q=1$ case of these polynomials appear in \cite{boson}, Equations (20) and (21). We define these polynomials recursively by
\begin{align}
\label{allqstir}
\allqstir{n}{k}{q}{a} &= \sum_{i=1}^{k}  q^{\binom{a-k+i}{2}} \qbinom{i}{a -k+i} \frac{[a]_{q}!}{[k-i]_{q}!} \allqstir{n-1}{i}{q}{a} \\
\allqstir{1}{k}{q}{a} &= \chi(k=a) .
\end{align}
Note that, at $a=1$, the recursion simplifies to the $q$-Stirling numbers $\qstir{n}{k}{q}$.

\begin{prop}
\label{prop:a}
When $\alpha = a^n$, 
\begin{align*}
\qmah{\alpha}{k}{q} &= \frac{\qint{k}!}{\left(\qint{a}!\right)^{n}} \allqstir{n}{k}{q}{a} .
\end{align*}
When $a=1$, this formula reduces to the formula
\begin{align*}
\qmah{n}{k}{q} &= \qint{k}! \qstir{n}{k}{q}.
\end{align*}
\end{prop}

\begin{proof}
We work by induction on $n$. When $n=1$, the right-hand side of Proposition \ref{prop:a} equals
\begin{align*}
\frac{\qint{k}!}{\left(\qint{a}!\right)^n} \allqstir{1}{k}{q}{a} = \chi(k=a)
\end{align*}
which is equal to the left-hand side by (\ref{initial-condition}). 

If $n > 1$, we use the induction hypothesis with the recursion (\ref{mah-rec}) to compute
\begin{align*}
 \qmah{\alpha}{k}{q} &= \sum_{\el=1}^{k} q^{\binom{a-k+\el}{2}} \qbinom{k}{a-k+\el,k-a,k-\el} \left( \frac{\qint{\el}!}{\left(\qint{a}\right)^{n-1}} \allqstir{n-1}{\el}{q}{a} \right) \\
&= \frac{\qint{k}!}{\left( \qint{a} !\right)^{n-1}} \sum_{\el=1}^{k} q^{\binom{a-k+\el}{2}} \frac{\qint{\el}!}{\qint{a-k+\el}! \qint{k-a}! \qint{k-\el}!} \allqstir{n-1}{\el}{q}{a} \\
&= \frac{\qint{k}!}{\left( \qint{a} !\right)^{n}} \sum_{\el=1}^{k} q^{\binom{a-k+\el}{2}} \qbinom{\el}{a-k+\el} \frac{\qint{a}!}{\qint{k-\el}!} \allqstir{n-1}{\el}{q}{a} \\
&= \frac{\qint{k}!}{\left( \qint{a} !\right)^n} \allqstir{n}{k}{q}{a} 
\end{align*}
by (\ref{allqstir}) with $i=\el$. 
\end{proof}

It would be interesting to give a more combinatorial proof of Proposition \ref{prop:a}, especially one that would shed light on why each of the three terms $\qint{k}!$, $(\qint{a}!)^n$, and $\allqstir{n}{k}{q}{a}$ appears.

\section{Extending Macdonald polynomials}
\label{sec:extended-mac}

In this subsection, we apply our $\inv$ and $\maj$ statistics to the combinatorial definition of Macdonald polynomials for hook shapes, as given (for any shape) in \cite{hhl}. This yields functions whose coefficients are four-variable polynomials instead of the usual two-variable polynomial coefficients. Our previous work allows us to prove that these polynomials are symmetric and to expand them into Schur functions. 

For convenience, for any statistic $\stat$, let 
\begin{align*}
\stat_{[a, b]}(\sg) = \stat(\sg_{a} \sg_{a+1} \ldots \sg_{b}).
\end{align*}
We will also need two new statistics, one on permutations and one on descent-starred permutations:
\begin{align*}
\rlmaj(\sg) &= \sum_{i \in \Des(\sg)} (n-i) \\
\op{\rlmaj}((\sg, S)) &= \rlmaj(\sg) - \sum_{i \in S} |\Des(\sg) \cap [1, i]|.
\end{align*}
We set $\widetilde{H}_{n, m}(x; q, t, u, v)$ equal to
\begin{align*}
&\sum_{\sg \in \{1,2,\ldots\}^n} q^{\inv_{[m+1, n]}(\sg)} t^{\maj_{[1, m]}(\sg)} x^{\sg}  \\
&\ \times \prod_{i \in \Des(\sg) \cap [m+1, n]} \left(1 + u/q^{\inv_{[m+1, n]}^{\Box, i}(\sg) + 1} \right) \prod_{j=1}^{|\Des(\sg) \cap [1, m]|} \left(1 + v/t^{j}\right) .
\end{align*}
We will refer to these polynomials as \emph{starred Macdonald polynomials of hook shape}. These polynomials can be thought of as a sum over all descent-starred multiset permutations $(\sg, S)$ in the Young diagram for the shape $(n-m, 1^{m})$ where we calculate our $\op{\maj}$ and $\op{\inv}$ statistics down the column and across the row, respectively. When $u = v = 0$, we obtain the (modified) Macdonald polynomial for the shape $(n-m, 1^{m})$, as proven in \cite{hhl}. Here is an example filling for $n=8$ and $m=3$.
\begin{center}
\begin{ytableau}
2 \\
_{\phantom{\ast}}5_{\ast} \\
3 \\
_{\phantom{\ast}}6_{\ast} & 1 & 7 & 4 & 8
\end{ytableau}
\end{center}
The numbers followed by stars are the starred descents. The weight of this filling would be $q^2 t$, since the descent-starred permutation $6_{\ast}1 7 4 8$ has 2 inversions, both ending at the 4, and $25_{\ast}36$ has major index equal to 1. 

Our main result in this section allows us to transfer many important properties of the Macdonald polynomials to the starred Macdonald polynomials of hook shape. The $u=v=0$ case of this result was originally proved in \cite{stem}. The result requires some definitions on standard Young tableaux. The \emph{descent set} of a standard Young tableau $T$ (in French notation) is the set of all $i$ such that $i+1$ is strictly north (and weakly west) of $i$ in $T$. Then, for any standard Young tableau $T$ with $n$ entries, 
\begin{equation*}
\begin{aligned}[c]
\maj(T) &= \sum_{i \in \Des(T)} i
\end{aligned}
\qquad
\begin{aligned}
\rlmaj(T) &= \sum_{i \in \Des(T)} (n - i) .
\end{aligned}
\end{equation*}
For example, here is a standard Young tableau with descent set $\{2, 5\}$. 
\begin{displaymath}
\begin{ytableau}
3 & 4 & 6 \\
1 & 2 & 5 & 7
\end{ytableau}
\end{displaymath}

\begin{thm}
\label{thm:mac}
The starred Macdonald polynomials of hook shape are symmetric. Furthermore, for $\lambda \vdash n$ the coefficient of the Schur function $s_{\lambda}(x)$ in \\ $\widetilde{H}_{n, m}(x; q, t, u, v)$ is equal to
\begin{align*}
&\sum_{T \in \syt(\lambda)} q^{\rlmaj_{[m+1, n]}(T)} t^{\maj_{[1, m]}(T)} \\
&\times \prod_{i = 1}^{|\Des(T) \cap [m+1, n]|} \left(1 + uq^{-i} \right)   \prod_{j=1}^{|\Des(T) \cap [1, m]|} \left(1 + vt^{-j}\right) .
\end{align*}
\end{thm}

\begin{proof}
We begin with a descent-starred multiset permutation $(\sg, S)$. Then we apply the bijection
\begin{align*}
\gamma = \text{complement} \circ \text{reverse} \circ \psi_{\beta,\el} \circ \text{reverse} \circ \text{complement}
\end{align*}
to the descent-starred multiset permutation $(\sg_{m+1}\ldots\sg_{n}, S \cap [m+1, n])$ for suitable $\beta, \el$.  Since preserves the rightmost letter, $\gamma$ preserves $\sg_{k+1}$ and sends the $\op{\inv}$ of $(\sg_{m+1}\ldots\sg_{n}, S \cap [m+1, n])$ to the $\op{\rlmaj}$ of the resulting descent-starred multiset permutation. Hence, $\widetilde{H}_{n, m}$ equals
\begin{align}
\label{mac-rlmaj}
&\sum_{\sg \in \{1,2,\ldots\}^n} x^{\sg} q^{\rlmaj_{[k+1, n]}(\sg)} t^{\maj_{[1, k]}(\sg)} \\
&\times \prod_{i = 1}^{|\Des(\sg) \cap [k+1, n]|} \left(1 + u/q^{i} \right) \prod_{j=1}^{|\Des(\sg) \cap [1, k]|} \left(1 + v/t^{j}\right)  .
\end{align}
For any composition $\alpha$, the coefficient of $x^{\alpha}$ in this expression is just the sum over $\sg$ that are permutations of the multiset $\{1^{\alpha_{1}}, 2^{\alpha_{2}}, \ldots \}$. We would like to show that the coefficients of $x^{\alpha}$ and $x^{\alpha^{(r)}}$ are equal, where $\alpha^{(r)}$ is obtained from $\alpha$ by switching $\alpha_{r}$ and $\alpha_{r+1}$. To do this, we apply a procedure known as \emph{$r$-pairing} to the permutation $\tau$. We illustrate $r$-pairing via the example in Figure \ref{fig:r-pairing}. To perform $r$-pairing on a sequence, we begin by temporarily ignoring all entries not equal to $r$ or $r+1$. Then we pair off adjacent occurrences of $r+1$ and $r$, ignoring previously paired entries and iterating this pairing procedure. When there are no more such pairs, we replace each un-paired occurrence of $r$ with an $r+1$ and vice versa. Finally, we re-insert the entries we had temporarily removed in their initial positions. We provide an example above with $r=2$.  

For our purposes, it is enough to know that $r$-pairing replaces $\sg$ with a permutation of the multiset $\{1^{\alpha_{1}}, \ldots, r^{\alpha_{r+1}}, (r+1)^{\alpha_{r}}, \ldots, \}$, and that this permutation has the same descent set as $\sg$. Therefore $r$-pairing does not alter any of the expressions in (\ref{mac-rlmaj}), so $\widetilde{H}_{n,m}$ is symmetric.

\begin{figure}
\begin{center}
\begin{align*}
&24231243331324123321 \\
&2\phantom{4}2\underline{3\phantom{1}2}\phantom{4}333\phantom{1}\underline{32}\phantom{41}23\underline{32}\phantom{1} \\
&2\phantom{4}2\phantom{\underline{3\phantom{1}2}}\phantom{4}33\underline{3\phantom{1}\phantom{\underline{32}}\phantom{41}2}3\phantom{\underline{32}}\phantom{1} \\
&2\phantom{4}2\phantom{\underline{3\phantom{1}2}}\phantom{4}33\phantom{\underline{3\phantom{1}\phantom{\underline{32}}\phantom{41}2}}3\phantom{\underline{32}}\phantom{1} \\
&3\phantom{4}3\phantom{\underline{3\phantom{1}2}}\phantom{4}22\phantom{\underline{3\phantom{1}\phantom{\underline{32}}\phantom{41}2}}2\phantom{\underline{32}}\phantom{1} \\
&3\phantom{4}3\phantom{\underline{3\phantom{1}2}}\phantom{4}22\underline{3\phantom{1}\phantom{\underline{32}\phantom{41}}2}2\phantom{\underline{32}}\phantom{1} \\
&3\phantom{4}3\underline{3\phantom{1}2}\phantom{4}223\phantom{1}\underline{32}\phantom{41}22\underline{32}\phantom{1} \\
&34331242231324122321
\end{align*}
\end{center}
 \caption{An example of $r$-pairing with $r=2$.}
\label{fig:r-pairing}
\end{figure}

Furthermore, if we apply the Robinson-Schensted-Knuth correspondence \cite{ec2} to each permutation $\sg$ involved in the coefficient of $x^{\lambda}$ for a partition $\lambda$, we see that the coefficient of $x^{\lambda}$ in $\widetilde{H}_{n,m}$ is
\begin{align*}
&K_{\lambda, (n-m, 1^m)} \sum_{T \in \syt(\lambda)} q^{\rlmaj_{[m+1, n]}(T)} t^{\maj_{[1, m]}(T)} \\
&\ \times \prod_{i = 1}^{|\Des(T) \cap [m+1, n]|} \left(1 + u/q^{i} \right) \prod_{j=1}^{|\Des(T) \cap [1, m]|} \left(1 + v/t^{j}\right) .
\end{align*}
Here $K_{\lambda, (n-m, 1^m)}$ is the Kostka number. Translating from the coefficient of $x^{\lambda}$ to the coefficient of $s_{\lambda}(x)$ exactly consists of removing this Kostka number, so the theorem follows.
\end{proof}

One consequence of Theorem \ref{thm:mac} is the identity
\begin{align}
\label{bijective-symmetry}
\widetilde{H}_{n, m}(x; q, t, u, v) &= \widetilde{H}_{n, n-m}(x; t, q, v, u).
\end{align}
This parallels the well-known identity
\begin{align}
\label{conjugate-id}
\tilde{H}_{\mu}(x;q,t) &= \tilde{H}_{\mu^{\prime}}(x;t,q)
\end{align}
for Macdonald polynomials. For Macdonald polynomials of hook shape, there is a direct bijective proof of \eqref{conjugate-id}. We would like to find a similar bijective proof of \eqref{bijective-symmetry}.

In future work, we hope to define and explored starred Macdonald polynomials for non-hook shapes. The full combinatorial formulation of Macdonald polynomials in \cite{hhl} essentially allows one to either star horizontally or vertically, giving analogs of Macdonald polynomials with three-variable polynomials for coefficients. At this point, we have been unable to define a four-variable generalization in the non-hook case that retains desirable properties such as symmetry. It would also be interesting to develop connections between these polynomials and Garsia-Haiman modules.

\bibliographystyle{alpha}
\bibliography{statistics}

\begin{thebibliography}{CPYY10}

\bibitem[BPS03]{boson}
P.~Blasiak, K.~A. Penson, and A.~I. Solomon.
\newblock The general boson normal ordering problem.
\newblock {\em Physics Letters A}, 309:198--205, 2003.

\bibitem[Car75]{carlitz}
L.~Carlitz.
\newblock A combinatorial property of $q$-{Eulerian} numbers.
\newblock {\em Amer. Math. Monthly}, 82:51--54, 1975.

\bibitem[CPYY10]{moon}
W.~Y.~C. Chen, S.~Poznanovik, C.~H. Yan, and A.~L.~B. Yang.
\newblock Major index for 01-fillings of moon polyominoes.
\newblock {\em J. Combin. Theory, Ser. A}, 117(8):1058--1081, 2010.

\bibitem[FH08]{foata-han}
D.~Foata and G.-N. Han.
\newblock {Fix-Mahonian} calculus, i: {Two} transformations.
\newblock {\em Europ. J. Combin.}, 29(7):1721--1732, 2008.

\bibitem[Foa68]{foata}
D.~Foata.
\newblock On the {Netto} inversion number of a sequence.
\newblock {\em Proc. Amer. Math. Soc.}, 19:236--240, 1968.

\bibitem[HHL05]{hhl}
J.~Haglund, M.~Haiman, and N.~Loehr.
\newblock A combinatorial formula for {Macdonald} polynomials.
\newblock {\em J. Amer. Math. Soc.}, 18:735--761, 2005.

\bibitem[HLR05]{hlr}
J.~Haglund, N.~Loehr, and J.~B. Remmel.
\newblock Statistics on wreath products, perfect matchings, and signed words.
\newblock {\em Europ. J. Combin.}, 26:835--868, 2005.

\bibitem[Mac15]{macmahon}
P.~A. MacMahon.
\newblock {\em Combinatory Analysis}, volume~1.
\newblock Cambridge University Press, 1915.

\bibitem[RW15]{rw}
J.~B. Remmel and A.~T. Wilson.
\newblock An extension of {MacMahon's} equidistribution theorem to ordered set
  partitions.
\newblock {\em J. Combin. Theory, Ser. A}, 134:242--277, August 2015.

\bibitem[Sta99]{ec2}
R.~P. Stanley.
\newblock {\em Enumerative Combinatorics}, volume~2.
\newblock Cambridge University Press, 1999.

\bibitem[Ste94]{stem}
J.~R. Stembridge.
\newblock Some particular entries of the two-parameter {Kostka} matrix.
\newblock {\em Proc. Amer. Math. Soc.}, 121:469--490, 1994.

\end{thebibliography}
\label{sec:biblio}

\end{document}